\documentclass[12pt]{amsart}

\makeatletter
\def\newrefformat#1#2{%
  \@namedef{pr@#1}##1{#2}}
\def\fref#1{\@prettyref#1:}
\def\@prettyref#1:#2:{%
  \expandafter\ifx\csname pr@#1\endcsname\relax%
    \PackageWarning{prettyref}{Reference format #1\space undefined}%
    \ref{#1:#2}%
  \else%
    \csname pr@#1\endcsname{#1:#2}%
  \fi%
}
\makeatother

\newrefformat{eq}{\textup{(\ref{#1})}}
\newrefformat{ax}{\textup{\ref{#1}}}
\newrefformat{chap}{Chapter~\ref{#1}}
\newrefformat{sec}{Section~\ref{#1}}
\newrefformat{apdx}{Appendix~\ref{#1}}
\newrefformat{tab}{Table~\ref{#1} on page~\pageref{#1}}
\newrefformat{fig}{Figure~\ref{#1} on page~\pageref{#1}}
\newrefformat{item}{\ref{#1}}

\usepackage{amsthm}
\usepackage{hyperref}

\newcommand{\mynewthm}[3][]{%
  \def\PARAM{#1}
  \ifx\PARAM\empty
  \newtheorem{#2}[dummythm]{#3}
  \else
  \newtheorem{#2}{#3}[#1]
  \fi
  \newtheorem*{#2*}{#3}%
  \newrefformat{#2}{#3~\ref{##1}}%
}

\theoremstyle{plain}
\mynewthm{thm}{Theorem}
\mynewthm{prp}{Proposition}
\mynewthm{lem}{Lemma}
\mynewthm{fct}{Fact}
\mynewthm{cor}{Corollary}

\theoremstyle{definition}
\mynewthm{dfn}{Definition}
\mynewthm{conv}{Convention}
\mynewthm{ntn}{Notation}
\mynewthm{conj}{Conjecture}
\mynewthm{cst}{Construction}

\theoremstyle{remark}
\mynewthm{rmk}{Remark}
\mynewthm{qst}{Question}
\mynewthm{exm}{Example}

\newcommand{\myenumlabel}[1]{\textnormal{(\roman{#1})}}

\usepackage{amsmath}
\usepackage{amssymb}
\usepackage{mathrsfs}

\renewcommand{\today}{%
  \number\day\space
  \ifcase\month\or
  January\or February\or March\or April\or May\or June\or
  July\or August\or September\or October\or November\or December\fi
  \space \number\year}

\newcounter{cycprfcnt}
\newenvironment{cycprf}%
{\begin{list}{\PackageWarning{begnac}{Label required for cycprf}}%
  {%
    \setcounter{cycprfcnt}{1}
    \setlength{\itemindent}{0.5\leftmargin}%
    \setlength{\leftmargin}{0pt}%
    \newcommand{\cpcurr}{\myenumlabel{cycprfcnt}}%
    \newcommand{\cpnext}{\addtocounter{cycprfcnt}{1}\cpcurr}%
    \newcommand{\cpnum}[1]{\setcounter{cycprfcnt}{##1}\cpcurr}%
    \newcommand{\cpfirst}{\cpnum{1}}%
    \newcommand{\impnext}{\cpcurr{} $\Longrightarrow$ \cpnext.}%
    \newcommand{\impfirst}{\cpcurr{} $\Longrightarrow$ \cpfirst.}%
  }%
}%
{\qedhere\end{list}}%

\makeatletter

\def\indsym#1#2{%
  \setbox0=\hbox{$\m@th#1x$}%
  \kern\wd0%
  \hbox to 0pt{\hss$\m@th#1\mid$\hbox to 0pt{$\m@th#1^{#2}$\hss}\hss}%
  \lower.9\ht0\hbox to 0pt{\hss$\m@th#1\smile$\hss}%
  \kern\wd0}

\def\nindsym#1#2{%
  \setbox0=\hbox{$\m@th#1x$}%
  \kern\wd0%
  \hbox to 0pt{\hss$\m@th#1\not$\kern1.4\wd0\hss}
  \hbox to 0pt{\hss$\m@th#1\mid$\hbox to 0pt{$\m@th#1^{#2}$\hss}\hss}%
  \lower.9\ht0\hbox to 0pt{\hss$\m@th#1\smile$\hss}%
  \kern\wd0}

\def\dotminussym#1#2{%
  \setbox0=\hbox{$\m@th#1-$}%
  \kern.5\wd0%
  \hbox to 0pt{\hss\hbox{$\m@th#1-$}\hss}%
  \raise.6\ht0\hbox to 0pt{\hss$\m@th#1.$\hss}%
  \kern.5\wd0}
\newcommand{\dotminus}{\mathbin{\mathpalette\dotminussym{}}}

\makeatother

\renewcommand{\emptyset}{\varnothing}
\renewcommand{\setminus}{\smallsetminus}
\def\models{\vDash}

\DeclareMathOperator{\tp}{tp}

\DeclareMathOperator{\tS}{S}

\DeclareMathOperator{\Conv}{Conv}

\newcommand{\fB}{\mathfrak{B}}

\newcommand{\cC}{\mathcal{C}}
\newcommand{\cD}{\mathcal{D}}

\newcommand{\cL}{\mathcal{L}}

\newcommand{\cM}{\mathcal{M}}

\newcommand{\cP}{\mathcal{P}}

\newcommand{\bE}{\mathbb{E}}

\newcommand{\bP}{\mathbb{P}}

\newcommand{\setR}{\mathbb{R}}
\newcommand{\setC}{\mathbb{C}}
\newcommand{\setN}{\mathbb{N}}

\newcommand{\setQ}{\mathbb{Q}}

\newcommand{\xx}{{\mathbf x}}

\newcommand{\KK}{{\mathbf K}}

\newcommand{\sfP}{{\mathsf P}}

\newcommand{\half}[1][1]{\hbox{$\frac{#1}{2}$}}

\begin{document}

\title{Continuous and random Vapnik-Chervonenkis classes}

\author{Ita\"\i{} \textsc{Ben Yaacov}}

\address{Ita\"\i{} \textsc{Ben Yaacov} \\
  Universit\'e de Lyon \\
  Universit\'e Lyon 1 \\
  Institut Camille Jordan, UMR 5208 CNRS \\
  43 boulevard du 11 novembre 1918 \\
  F-69622 Villeurbanne Cedex \\
  France}

\urladdr{http\string://math.univ-lyon1.fr/\textasciitilde begnac/}

\thanks{Research initiated during the workshop ``Model theory of
  metric structures'', American Institute of Mathematics Research
  Conference Centre, 18 to 22 September 2006}
\thanks{Research supported by
  ANR chaire d'excellence junior (projet THEMODMET) and
  by the European Commission Marie Curie Research Network ModNet}

\date{\today}

\begin{abstract}
  Nous d\'emontrons que si $T$ est une th\'eorie d\'ependante, sa randomis\'ee
  de Keisler $T^R$ l'est aussi.

  Pour faire cela nous g\'en\'eralisons la notion d'une classe de
  Vapnik-Chervonenkis \`a des familles de fonctions \`a valeurs dans
  $[0,1]$ (une classe de Vapnik-Chervonenkis \emph{continue}),
  et nous caract\'erisons les familles de fonctions ayant cette
  propri\'et\'e par la vitesse de croissance de la largeur moyenne d'une
  famille de compacts convexes associ\'es.
\end{abstract}

\maketitle

In this paper we answer a question lying at the intersection of
two currently active research themes in model theory.

The first theme is that of dependent theories, i.e., first
order theories which do no possess the independence property,
first defined by Shelah \cite{Shelah:1971-StabilityAndFCP}.
A formula $\varphi(x,y)$ is said to have the \emph{independence property},
or to be \emph{independent},
in a theory $T$ if for every $n$ one can find a model
$\cM \models T$, $b_i \in M$ for $i < n$, and $a_w \in M$ for $w \subseteq n$, such that
$\cM \models \varphi(a_w,b_i) \Longleftrightarrow i \in w$.
The theory $T$ has the independence property if at least one formula
has it in $T$; equivalently, a theory $T$ is dependent if every
formula is.
A stable theory is necessarily dependent.
(More generally, a theory $T$ is unstable if and only if it is
independent or has the \emph{strict order property}.
See \cite[Th\'eor\`eme~12.38]{Poizat:Cours} for the proof in classical
first order logic.
It can be adapted easily to continuous logic following standard
translation methods.)

The recent use of properties of dependent theories  for the
solution of the so-called Pillay Conjecture in
\cite{Hrushovski-Peterzil-Pillay:GroupsMeasureNIP} earned them a
considerable increase in general interest.
It should be pointed out that some refer to dependent theories as
NIP (Non Independence Property) theories.
Since non-NIP theories are quite ``wild'', research in this area
concentrates on theories which are not non-NIP.

The second theme of research is that of continuous logic and metric
structures.
Studying metric structures using a model theoretic approach dates back
to  Henson \cite{Henson:NonstandardHulls} and
Krivine and Maurey \cite{Krivine-Maurey:EspacesDeBanachStables}.
Continuous first order logic was much more
recently introduced in \cite{BenYaacov-Usvyatsov:CFO} as a formalism
for this study, shifting the point of view much closer to
classical first order logic.
The independence property has a natural analogue for metric
structures, and one may speak of dependent continuous theories.
In contrast with the body of work concerning classical dependent
theories, to the best of our knowledge dependent continuous theories
have hardly been studied to date.

The intersection of these two themes in which we are interested stems
from H.\ Jerome Keisler's
randomisation construction.
This first appeared in \cite{Keisler:Randomizing}, where to every
(complete) first order theory $T$ he associated the first order theory
of spaces of random variables in models of $T$.
Since classical first order logic is not entirely adequate for the
treatment of spaces of random variables, which are metric by nature,
this construction was subsequently improved
to produce for every theory $T$ the \emph{continuous} theory
$T^R$ of spaces of random variables taking values in models of $T$
(see \cite{BenYaacov-Keisler:MetricRandom}).
While we shall not go through the details of the construction, we shall
point the main properties of $T^R$ in \fref{sec:DependentTheories}.
The author has shown that
\begin{enumerate}
\item The randomisation of a stable theory is stable
  (\cite{BenYaacov-Keisler:MetricRandom}, in preparation).
\item On the other hand, the randomisation of a simple unstable theory
  is not simple.
  More generally, the randomisation of an independent theory cannot be
  simple, and is generally wild.
\end{enumerate}
In other words, independent theories are somehow wild with respect
to randomisation (even if they do satisfy other tameness properties
such as simplicity), while stable theories are tame.
It is natural to ask whether the dividing line for tame randomisation
lies precisely between dependent and independent theories.
A more precise instance of this question would be,
is the randomisation of a dependent theory dependent?
While Keisler's construction was only stated for a complete classical
theory $T$ it can be carried out just as well for an arbitrary
continuous theory
$T$ (with some minor technical changes) and the question can also be
posed when $T$ is continuous.

The article
\cite{Hrushovski-Peterzil-Pillay:GroupsMeasureNIP} mentioned above
relates dependent theories and probability measures on types, i.e.,
with types in the randomised theory.
This suggested that dependent theories should be tame with respect to
randomisation, so  the answer to the question above should be positive.
(To our best recollection this was first conjectured by Anand Pillay
in the AIM workshop on Model Theory of Metric Structures, September 2006).

In order to give a positive answer we shall consider some
purely combinatorial aspects of the independence property,
observed by Shelah \cite{Shelah:1971-StabilityAndFCP}
and independently by Vapnik and Chervonenkis
\cite{Vapnik-Chervonenkis:UniformConvergence}
and seek to prove they extend to the continuous setting.
(The connection between dependent theories and
the work of Vapnik and Chervonenkis was pointed out by Laskowski
\cite{Laskowski:VapnikChervonenkisClasses}.)
Doing so we will need a new means for measuring the size of a set, as
merely counting points will no longer do.
The Gaussian mean width turns out to serve our purposes quite well
(the Lebesgue measure of the set once inflated a little will also be
useful, but to a much lesser extent).
The Gaussian mean width commutes, in a sense, with the
randomisation construction, and it follows painlessly that the
randomisation of a dependent relation is again a dependent relation.
Thus the technical core of this paper has nothing to do with model
theory and deals rather with
combinatorics and the geometry of convex compacts.

\smallskip

\fref{sec:MeanWidth} consists of a few basic facts regarding convex
compacts in $\setR^n$ and their mean width.

The combinatorial core of the article is in
\fref{sec:FuzzyContinuousVC} where continuous Vapnik-Chervonenkis
classes and dependent relations are characterised via the
growth rate of the mean width of an associated family of sets.

\fref{sec:Crush} consists of a technical interlude where we prove that
continuous combinations of dependent relations are dependent.

In \fref{sec:Randomise} we consider random dependent
relations.
Using the mean width criterion from
\fref{sec:FuzzyContinuousVC}
we show that if a random family of functions is uniformly dependent
then its expectation is dependent as well.

The proper model theoretic contents of this paper is restricted to
\fref{sec:DependentTheories}.
We define dependent theories and the randomisation of a theory.
The main theorem, asserting that the randomisation of a dependent
theory is dependent, follows easily from earlier results.
We also extend to continuous logic a classical result saying
that in order to verify that
a theory is dependent it suffices to verify that every
formula $\varphi(x,\bar y)$ is dependent where $x$ is a single variable.

\section{Facts regarding convex compacts and mean width}
\label{sec:MeanWidth}

This section contains few properties of the mean width function.
The author is much indebted to Guillaume Aubrun for having introduced
him to this notion and its properties.
All the results presented here are easy to verify and are
either folklore (see for example
\cite{Aubrun-Szarek:TensorProductsAndVolume}) or (in the
case of integrals of convex compacts, as far as we know)
minor generalisations thereof.

Let $A \subseteq \setR^n$ a bounded set.
For $y \in \setR^n$ define $h_A(y) = \sup_{x \in A} \langle x,y\rangle$.
As a function of $y$, $h_A$ is positively homogeneous and
sub-additive, and thus in particular convex.

Let $u \in S^{n-1}$.
The real numbers
$t_1 = h_A(u)$ and $t_2 = -h_A(-u)$ are then minimal and maximal,
respectively, so that $t_2 \leq \langle x,u\rangle \leq t_1$ for all $x\in A$,
i.e., such that $A$ lies
between the two hyperplanes $t_2u + u^\perp$ and $t_1u + u^\perp$.
The \emph{width} of $A$ in the direction $u \in S^{n-1}$ is therefore
defined to be $w(A,u) = h_A(u) + h_A(-u)$.

Let $K = \overline{\Conv}(A)$ be the closed convex envelope of $A$,
i.e., the intersection of all closed half-spaces containing $A$.
Then $h_K = h_A$ and:
\begin{gather*}
  K = \bigcap_{u\in S^{n-1}} \{x\colon \langle x,u\rangle \leq h_K(u)\}.
\end{gather*}
We may thus identify a convex compact $K \subseteq \setR^n$ with
$h_K\colon S^{n-1} \to \setR$.
In this case the supremum in the definition of $h_K$ is attained at an
extremal point of $K$.

It is easy to observe that the function $h_K(u)$ is monotone,
positively homogeneous and additive in $K$, i.e.,
$K \subseteq K' \Longrightarrow h_K(u) \leq h_{K'}(u)$, $h_{\alpha K}(u) = \alpha h_K(u)$ for $\alpha \geq 0$ and
\begin{align*}
  h_{K+K'}(u)
  & = \max_{x\in K+K'} \langle x,u\rangle
  = \max_{y\in K,z\in K'} \langle y+z,u\rangle \\
  & = \max_{y\in K} \langle y,u\rangle + \max_{z\in K'} \langle z,u\rangle
  = h_K(u) + h_{K'}(u).
\end{align*}

Let $(X,\fB,\mu)$ be a measure space,
$\KK$ a mapping from $X$ to the space of convex compacts in $\setR^n$.
Say that $\KK$ is \emph{measurable} (respectively, integrable) if
$\omega\mapsto h_{\KK(\omega)}(u)$ is for all $u \in S^1$.
Notice that $u \mapsto h_{\KK(\omega)}(u)$ is $b(\KK(\omega))$-Lipschitz where
$b(K) = \max \{h_K(u)\colon u \in S^{n-1}\}$.
Thus, if $h_\KK(u)$ is measurable for all $u$ in some dense
(and possibly countable) subset of $S^{n-1}$ then
$b(\KK)$ is measurable and thus $\KK$ is.
If $\KK$ is integrable define:
\begin{gather*}
  h(u) = \int h_\KK(u)\, d\mu, \\
  K = \int \KK \, d\mu = \bigcap_{u \in S^{n-1}} \{x\colon \langle x,u\rangle \leq h(u)\}.
\end{gather*}
Clearly $K$ is a convex compact, and if
$\xx\colon X \to \setR^n$ satisfies $\xx(\omega) \in \KK(\omega)$ a.e.\ then
$\int \xx\, d\mu \in \int \KK \, d\mu$.
We claim furthermore that $h_K = h$.
Indeed, it is clear by definition of $K$ that $h_K \leq h$.
Conversely, given $u \in S^{n-1}$  we may complete it to an orthonormal
basis $u_0 = u, u_1, \ldots, u_n$.
For each $\omega \in X$ there is a unique $x_\omega \in \KK(\omega)$ such that
the tuple $(\langle x_\omega,u_0\rangle,\ldots,\langle x_\omega,u_{n-1}\rangle)$ is maximal in lexicographical
order (among all $x \in \KK(\omega)$).
In particular $\langle x_\omega,u\rangle = h_{\KK(\omega)}(u)$.
Moreover, the mapping $\xx\colon \omega \mapsto x_\omega$ is measurable,
$x = \int \xx\,d\mu \in K$ and
$h_K(u) \geq \langle x,u\rangle = h(u)$.
Thus $h_K = h$ as required.

The \emph{mean width} of $K$ is classically defined as:
\begin{gather*}
  w(K) = \int_{S^{n-1}} w(K,u) d\mu = 2\int_{S^{n-1}} h_K(u) d\sigma,
\end{gather*}
where $\sigma$ is the normalised Lebesgue measure on the sphere.

\begin{lem}
  \label{lem:MeanWidthProperties}
  The mean width is a monotone, additive, positively homogeneous
  function of compact convex subsets of $\setR^n$.
  Moreover, if $\KK$ is an integrable family of convex compacts then
  \begin{gather*}
    w\left( \int \KK\, d\mu \right) = \int w(\KK)\,d\mu.
  \end{gather*}
\end{lem}
\begin{proof}
  Monotonicity of $w$ follows from monotonicity (in $K$) of
  $h_K$.
  Additivity and positive homogeneity are special cases of the
  summability which follows from earlier observations via Fubini's
  Theorem.
\end{proof}

As it happens it will be easier to calculate the following
variant of the mean width:
\begin{dfn}
  The \emph{Gaussian mean width} of a convex compact $K$ is defined as
  \begin{gather*}
    w_G(K) = \bE[w(K,G_n)] = 2\bE[h_K(G_n)],
  \end{gather*}
  where $G_n \sim N(0,I_n)$ (i.e., $G_n = (g_0,\ldots,g_{n-1})$ where $g_0,\ldots,g_{n-1}$ are
  independent random variables, $g_i \sim N(0,1)$).
\end{dfn}

Since the
distribution of $N(0,I_n)$ is rotation-invariant, the random variables
$\|G_n\|_2$ and $\frac{G_n}{\|G_n\|_2}$ are independent.
Let $\gamma_n = \bE[\|G_n\|_2]$.
We obtain:
\begin{gather*}
  w_G(K) = \bE[\|G_n\|_2w(K,G_n/\|G_n\|_2)] = \bE[ \|G_n\|_2 ] \bE[w(K,G_n/\|G_n\|_2)]
  = \gamma_n w(K).
\end{gather*}
One can further calculate that
\begin{gather*}
  \gamma_n =
  \sqrt{2}\frac{\Gamma\left(\frac{n+1}{2}\right)}{\Gamma\left(\frac{n}{2}\right)}.
  \intertext{As $\Gamma$ is $\log$-convex we obtain:}
  \gamma_n \leq
  \sqrt{2\frac{\Gamma\left(\frac{n}{2}+1\right)}{\Gamma\left(\frac{n}{2}\right)}}
  = \sqrt{2\frac{n}{2}} = \sqrt{n}, \\
  \gamma_n \geq
  \sqrt{2\frac{\Gamma\left(\frac{n+1}{2}\right)}{\Gamma\left(\frac{n-1}{2}\right)}}
  = \sqrt{2\frac{n-1}{2}} = \sqrt{n-1}.
\end{gather*}
Whence:
\begin{gather*}
  \sqrt{n-1} \leq \gamma_n \leq \sqrt{n}
\end{gather*}
Thus for example, if $B^n$ is the unit ball in $\setR^n$ then
\begin{gather*}
  w_G(B^n) = \gamma_n w(B^n) = 2\gamma_n \approx 2\sqrt{n}.
\end{gather*}

\begin{lem}
  \label{lem:GaussianMeanWidthProperties}
  The Gaussian mean width is a monotone, additive, positively
  homogeneous, function of compact convex subsets of $\setR^n$,
  and for an integrable family $\KK$:
  $w_G\left( \int \KK\, d\mu \right) = \int w_G(\KK)\,d\mu$.
\end{lem}
\begin{proof}
  Follows from \fref{lem:MeanWidthProperties} (or is proved
  identically).
\end{proof}

Let us calculate the mean width of the cube $[-1,1]^n$.
The maximum $h_K(y) = \max_{x\in[-1,1]^n} \langle x,y\rangle$ is always attained at an
extremal point, i.e.,
$h_K(y) = \max_{x \in \{-1,1\}^n} \langle x,y\rangle = \|y\|_1$.
Thus:
\begin{gather*}
  w_G([-1,1]^n) = 2\bE[ \|G_n\|_1 ]
  = 2n \bE[ |G_1| ] = 2n\sqrt{\frac{2}{\pi}}.
\end{gather*}
Thus, for $\varepsilon > 0$ we have:
\begin{gather}
  \label{eq:CubeMeanWidth}
  w_G([0,\varepsilon]^n) = \varepsilon n\sqrt{\frac{2}{\pi}}.
\end{gather}

\section{Fuzzy and continuous Vapnik-Chervonenkis classes}
\label{sec:FuzzyContinuousVC}

Let us start with a few reminders regarding the Vapnik-Chervonenkis
classes.
We shall follow Chapter 5 of van den Dries
\cite{vandenDries:oMinimal}.

Let us fix a set $X$ and a family of subsets $\cC \subseteq \cP(X)$.
Recall that $[X]^n$ denotes the collection of all subsets
of $X$ of size $n$, and let
$\cP^f(X) = \bigcup_{n<\omega} [X]^n$ denote the collection of finite subsets
of $X$.
For $F \in \cP^f(X)$ and $n < \omega$ let:
\begin{gather*}
  \cC \cap F = \{C\cap F\colon C \in \cC\}, \\
  f_\cC(n) = \max \{|\cC \cap F|\colon F \in [X]^n\}.
\end{gather*}

Clearly, $f_\cC(n) \leq 2^n$.
Define the \emph{Vapnik-Chervonenkis index} of $\cC$, denoted
$VC(\cC)$, to be the minimal $d$ such that $f_\cC(d) < 2^d$,
or infinity if no such $d$ exists.
If $VC(\cC) < \infty$ then $\cC$ is a \emph{Vapnik-Chervonenkis class}.

Let $p_d(x) = \sum_{k<d} \binom{x}{k} \in \setQ[x]$,
observing this is a polynomial of degree $d-1$.
\begin{fct}
  \label{fct:VapnikChervonenkisIndexBound}
  If $d = VC(\cC) < \infty$ then
  $f_\cC(n) \leq p_d(n)$ for all $n$.
\end{fct}

This can be viewed as a dichotomy result:
either $|\cC \cap F|$ is maximal (given $|F|$) for
arbitrarily large finite $F \subseteq X$, or it is always quite small
(polynomial rather than exponential).
It follows immediately from the following.

\begin{fct}
  \label{fct:VapnikChervonenkisLemma}
  Let $F$ be a finite set, $n = |F|$, and say
  $\cD \subseteq \cP(F)$ is such that $|\cD| > p_d(n)$.
  Then $F$ admits a subset $E \subseteq F$, $|E| = d$ such that
  $|\cD \cap E| = 2^d$.
\end{fct}

See \cite[Chapter~5]{vandenDries:oMinimal} for the proof, which
is attributed independently to
Shelah \cite{Shelah:1971-StabilityAndFCP}
and to Vapnik and Chervonenkis
\cite{Vapnik-Chervonenkis:UniformConvergence}.
This will also follow as a special case of a result we prove below.

Let us now add a minor twist to the setting, whose motivation will become
clear later on.
We allow the class $\cC$ to contain \emph{fuzzy subsets} of $X$,
i.e., objects $C$ such that for each $x \in X$ at most one of
$x \in C$ or $x \notin C$ holds, but possibly neither (in which case it is
not known whether $x$ belongs to $C$ or not).
This can be formalised by pair $C = (C_1,C_2)$ where $C_1, C_2 \subseteq X$
are disjoint, $C_1 = \{x \in X\colon x \in C\}$, $C_2 = \{x \in X\colon x \notin C\}$.

Let $C \sqsubseteq X$ denote that $C$ is a fuzzy subset of
$X$ and let $\sfP(X)$ denote the collection of fuzzy subsets.
If $F \subseteq X$, we say that $C$ \emph{determines a subset of $F$} if for
all $x \in F$ one of $x \in C$ or $x \notin C$ does hold, in which case
we define $C\cap F$ as usual, and otherwise we define $F\cap C = *$.
We may then define
\begin{gather*}
  \cC \cap F = \{C\cap F\colon C \in \cC\} \setminus \{*\}, \\
  f_\cC(n) = \max \{|\cC \cap F|\colon F \in [X]^n\}.
\end{gather*}
Thus $\cC \cap F$ is the collection of all subsets of $F$ which members
of $\cC$ determine.
Vapnik-Chervonenkis classes of fuzzy subsets of $X$ and the
corresponding index are defined as above, and the standard
proofs of \fref{fct:VapnikChervonenkisIndexBound}
and of \fref{fct:VapnikChervonenkisLemma} hold verbatim.

Our source for classes of fuzzy subsets of $X$ will be the following.
Let $Q \subseteq [0,1]^X$ be a collection of functions from $X$ to $[0,1]$.
For $0 \leq r < s \leq 1$ and $q \in Q$ we define a fuzzy set
$q_{r,s} \sqsubseteq X$ as follows: $x \in q_{r,s}$ if $q(x) \geq s$,
$x \notin q_{r,s}$ if $q(x) \leq r$, and it is unknown whether $x$ belongs to
$q_{r,s}$ or not if $r < q(x) < s$.
We define $Q_{r,s} = \{q_{r,s} \colon q \in Q\} \subseteq \sfP(X)$.
We say that $Q$ is a Vapnik-Chervonenkis class if $Q_{r,s}$ is for
every $0 \leq r < s \leq 1$.
Of course the index may vary with $r,s$.
However an easy argument shows that if $Q$ is a Vapnik-Chervonenkis
class then for every $\varepsilon > 0$ there exists $d(\varepsilon) < \omega$ which is an upper
bound for the Vapnik-Chervonenkis indexes of the classes
$Q_{r,r+\varepsilon}$ as $r$ varies in $[0,1-\varepsilon]$.
Notice that in the original case where $\cC \subseteq \cP(X)$,
if $Q = \{\chi_C\colon C \in \cC\}$ is the collection of characteristic
functions of members of $\cC$ then
$Q_{r,s} = \cC$ for every $0 \leq r < s \leq 1$, so the subset case is a
special case of the function case.

\begin{dfn}
  Let $0 \leq r_i < s_i \leq 1$ be given for $i < n$ and let
  $A \subseteq [0,1]^n$.
  We say that $A$ \emph{determines} a subset $w \subseteq n$ between
  $\bar r$ and $\bar s$ if there is a point $\bar a \in A$ such that
  $i \in w \Longrightarrow a_i \geq s_i$ and $i \notin w \Longrightarrow a_i \leq r_i$ for all $i < n$.
  In case $r_i = r$ and $s_i = s$ for all $i < n$ we say that
  $A$ determines $w$ between $r$ and $s$.

  We say that $A$ \emph{determines a $d$-dimensional $\varepsilon$-box}
  If $\varepsilon > 0$, $d \leq n$, and there are
  $i_0 < \ldots < i_{d-1} < n$ and
  $\bar r \in [0,1-\varepsilon]^d$ such that $\pi_{\bar i}(A) \subseteq [0,1]^d$ determines
  every subset of $d$ between $\bar r$ and $\bar r + \varepsilon$.

  Finally for $\varepsilon \geq 0$ we say that
  $A$ \emph{determines a strict $d$-dimensional $\varepsilon$-box}
  if it determines a $d$-dimensional $\varepsilon'$-box for some $\varepsilon' > \varepsilon$.
\end{dfn}

Thus if $F = \{x_1,\ldots,x_n\} \subseteq X$ then
$Q_{r,s} \cap F$ is in bijection with the subsets of $n$ determined by
$Q(\bar x)$ between $r$ and $s$.

Let us now relate this to the previous section.
Let again $Q \subseteq [0,1]^X$ be a collection of functions.
For a tuple $\bar x \in X^n$ and $q \in Q$ define:
\begin{gather*}
  q(\bar x) = (q(x_0),\ldots,q(x_n)) \in [0,1]^n, \\
  Q(\bar x) = (q(\bar x)\colon q \in Q), \\
  g_Q(n) = \sup \{ w_G(Q(\bar x))\colon \bar x \in X^n \}.
\end{gather*}

\begin{lem}
  \label{lem:DetermineImpliesWide}
  If $A \subseteq [0,1]^n$ determines an $n$-dimensional
  $\varepsilon$-box then $w_G(A) \geq \varepsilon n\sqrt{\frac{2}{\pi}}$.
  If $A$ determines a strict $n$-dimensional
  $\varepsilon$-box then $w_G(A) > \varepsilon n\sqrt{\frac{2}{\pi}}$.
\end{lem}
\begin{proof}
  It suffices to prove the first assertion.
  In this case there are
  $\bar r \in [0,1-\varepsilon]^n$ and for every
  $w \subseteq n$ there is $a_w \in A$ such that
  $a_w(i) \geq r + \varepsilon$ if $i \in w$ and $a_w(i) \leq r$ otherwise.
  Thus $A \supseteq (a_w\colon w \subseteq n) \supseteq \prod [r_i,r_i+\varepsilon] = \bar r_i + [0,\varepsilon]^n$.
  It follows that
  \begin{gather*}
    w_G(A) \geq w_G([0,\varepsilon]^n) = \varepsilon n\sqrt{\frac{2}{\pi}}.
    \qedhere
  \end{gather*}
\end{proof}
In other words, if $A$ determines an $n$-dimensional $\varepsilon$-box then
$\Conv(A)$ contains a set of the form
$\bar r + [0,\varepsilon]^n$.
The converse does not hold in general.

\begin{prp}
  \label{prp:MeanWidthImpliesVC}
  If $Q \subseteq [0,1]^X$ and $\lim \frac{g_Q(n)}{n} = 0$ then
  $Q$ is a Vapnik-Chervonenkis class.

  Moreover, for any function $g(n)$ such that
  $\lim \frac{g(n)}{n} = 0$ and any $\varepsilon > 0$ there is
  $d(g,\varepsilon) < \omega$ such that for any $Q \subseteq [0,1]^X$, if
  $g_Q \leq g$ then $d(g,\varepsilon) \geq VC(Q_{r,r+\varepsilon})$ for all $0 \leq r \leq 1-\varepsilon$.
\end{prp}
\begin{proof}
  Let $g = g_Q$ and $\varepsilon > 0$ be given, and find $n$ such that
  $\frac{g(n)}{n} < \varepsilon\sqrt{\frac{2}{\pi}}$.
  We claim that $d(g,\varepsilon) = n$ will do.

  Indeed, assume not.
  Then there are $r \in [0,1-\varepsilon]$ and $F = \{x_0,\ldots,x_{n-1}\} \subseteq X$
  such that $|Q_{r,r+\varepsilon} \cap F| = 2^n$, i.e., such that
  $Q(\bar x)$ determines an $n$-dimensional $\varepsilon$-box.
  By \fref{lem:DetermineImpliesWide} we
  have $g(n) \geq g_Q(n) \geq w_G(Q(\bar x)) \geq \varepsilon n\sqrt{\frac{2}{\pi}}$,
  a contradiction.
\end{proof}

For the converse a little more work is required.
Let $\pi\colon \setR^n \to \setR^{n-1}$ be the projection on the first $n-1$
coordinates.
For $A \subseteq \setR^n$ and $a \in \setR$ let
$A_{\leq a} = A \cap (\setR^{n-1} \times {]-}\infty,a])$,
$A_{> a} = A \cap (\setR^{n-1} \times {]}a,+\infty[)$.

Let $\lambda$ denote the Lebesgue measure.
\begin{lem}
  \label{lem:StepVCS}
  Let $A \subseteq [0,\ell+1]^n$ be a Borel set,
  $\lambda(A) > \ell^dp_d(n)$.
  Then at least one of the following holds:
  \begin{enumerate}
  \item $\lambda(\pi A) > \ell^dp_d(n-1)$.
  \item There is $a \in [0,\ell+1]$ such that
    $\lambda( \pi A_{\leq a} \cap \pi A_{> a+1}) > \ell^{d-1}p_{d-1}(n-1)$.
  \end{enumerate}
\end{lem}
\begin{proof}
  For $x\in [0,\ell+1]^{n-1}$ let
  \begin{gather*}
    A' = \{(x,a) \in [0,\ell+1]^{n-1}\times[0,\ell]\colon x \in \pi A_{\leq a} \cap \pi A_{> a+1}\} \\
    f_A(x) = \int \chi_A(x,y) \,dy,\qquad f_{A'}(x) = \int \chi_{A'}(x,y) \,dy.
  \end{gather*}
  Notice that $f_{A'}(x) + \chi_{\pi A}(x) \geq f_A(x)$, integrating which yields:
  \begin{gather*}
    \lambda(A') + \lambda(\pi(A)) \geq \lambda(A) > \ell^dp_d(n).
  \end{gather*}
  Recall that $p_d(n) = p_{d-1}(n-1) + p_d(n-1)$ and
  assume that the first case fails, i.e.,
  that $\lambda(\pi A) \leq \ell^dp_d(n-1)$.
  Then:
  \begin{gather*}
    \lambda(A') + \ell^dp_d(n-1) > \ell^dp_d(n) = \ell^dp_{d-1}(n-1) + \ell^dp_d(n-1)
    \intertext{whereby:}
    \lambda(A') > \ell p_{d-1}(n-1).
  \end{gather*}
  Then there is $a$ such that
  $\lambda(\{x\colon (x,a) \in A'\}) > p_{d-1}(n-1)$, which is precisely the second case.
\end{proof}

\begin{lem}
  \label{lem:VCS}
  Let $A \subseteq [0,\ell+1]^n$ be a Borel set, $\lambda(A) > \ell^dp_d(n)$.
  Then $A$ determines a strict $d$-dimensional $1$-box.
\end{lem}
\begin{proof}
  Follows immediately by induction on $n$ using the previous Lemma
  for the induction step.
\end{proof}

\begin{lem}
  \label{lem:VCSBox}
  Let $A \subseteq [0,1]^n$ and $c > 0$, $\varepsilon \geq 0$ be such that
  \begin{gather*}
    \lambda(A + [0,c]^n) > ( c+\varepsilon )^{n-d} (1-\varepsilon)^d p_d(n).
  \end{gather*}
  Then $A$ determines a strict $d$-dimensional $\varepsilon$-box.
\end{lem}
\begin{proof}
  Let $\ell = \frac{1-\varepsilon}{c+\varepsilon}$,
  so $\ell+1 = \frac{1+c}{c+\varepsilon}$.
  Let $B = (c+\varepsilon)^{-1} (A + [0,c]^n)$.
  Then $B \subseteq [0,\ell+1]^n$ and $\lambda(B) > \ell^dp_d(n)$.
  By \fref{lem:VCS} $B$ determines a strict
  $d$-dimensional $1$-box.
  Thus $A+[0,c]^n$ determines a strict $d$-dimensional
  $(c+\varepsilon)$-box, and
  $A$ determines a strict $d$-dimensional $\varepsilon$-box.
\end{proof}

In case $A \subseteq \{0,1\}^n$ then
$\lambda(A+[0,1]^n) = |A|$.
If in addition $|A| > p_d(n) =  (1+0)^{n-d}(1-0)^dp_d(n)$
then $A$ determines a (strict) $d$-dimensional $0$-box, i.e., an $\varepsilon$-box
for some arbitrarily small $\varepsilon > 0$.
But given that $A \subseteq \{0,1\}^d$ this is only possible if $A$ determines a
$d$-dimensional $1$-box.
Thus \fref{fct:VapnikChervonenkisLemma} follows as a special case of
\fref{lem:VCSBox}.

Now let us show that if
$\lambda(A+[0,c]^n)$ is small then $A$ is small in a different way, namely
has small Gaussian mean width.

\begin{lem}
  \label{lem:GaussianMeanWidthMeasureBound}
  Let $A \subseteq [0,1]^n$, $c > 0$.
  Then
  \begin{gather*}
    w_G(A) \leq (1+c)\sqrt{2n\log(\lambda(A+[0,c]^n)/c^n)}.
  \end{gather*}
\end{lem}
\begin{proof}
  Let us first observe that if $\varphi\colon\setR^n \to \setR$ is a linear functional then
  the convexity of the exponential function implies:
  \begin{align*}
    \exp(\varphi(x))
    & = \exp\left(
      \int_{y\in x+[-c,c]^n} \varphi(y)\,(2c)^{-n}d\lambda(y)
    \right) \\
    & \leq (2c)^{-n} \int_{x+[-c,c]^n}\exp( \varphi(y))\, d\lambda(y).
  \end{align*}
  Let $A' = 2A-1 \subseteq [-1,1]^n$
  and $B' = A' + [-c,c]^n = 2(A+[0,c]^n) - (1+c)$.
  Then by the previous observation we have:
  \begin{align*}
    \sup\{ \exp(\varphi(x))\colon x \in A' \}
    & \leq (2c)^{-n} \int_{B'}\exp(\varphi(x))\, d\lambda(x).
  \end{align*}
  Let $\beta > 0$ be an arbitrary parameter for the time being.
  For a fixed $x \in \setR^n$ we have $\beta\langle x,G_n\rangle \sim N(0,\beta^2\|x\|^2)$, and a
  straightforward calculation yields
  $\bE[\exp( \beta\langle x,G_n\rangle ) ] = \exp( \beta^2\|x\|^2/2 )$.
  Using concavity of the logarithm we obtain:
  \begin{align*}
    w_G(A')
    & = 2\bE\left[ \sup\left\{ \langle x,G_n\rangle \colon x \in A' \right\} \right] \\
    & = \frac{2}{\beta}\bE\left[ \log\left(
        \sup\left\{ \exp(\beta\langle x,G_n\rangle) \colon x \in A' \right\}
      \right) \right] \\
    & \leq \frac{2}{\beta}\log\left( \bE\left[
        (2c)^{-n} \int_{B'}\exp(\beta\langle x,G_n\rangle)\, d\lambda(x)
      \right] \right) \\
    & = \frac{2}{\beta}\log\left( (2c)^{-n} \int_{B'}
      \bE\left[ \exp(\beta\langle x,G_n\rangle) \right]\, d\lambda(x)
      \right) \\
    & = \frac{2}{\beta}\log\left( (2c)^{-n} \int_{B'}
      \exp\left( \frac{\beta^2\|x\|^2}{2} \right)\, d\lambda(x)
    \right) \\
    & \leq \frac{2}{\beta}\log\left(
      (2c)^{-n} \lambda(B') \exp\left( \frac{\beta^2(1+c)^2n}{2} \right)
    \right) \\
    & = \frac{2\log(\lambda(B')/(2c)^n)}{\beta}
    + \beta(1+c)^2n.
  \end{align*}
  Minimum is attained when
  $\beta = \frac{\sqrt{2\log(\lambda(B')/(2c)^n)}}{(1+c)\sqrt{n}}$, and substituting
  we obtain:
  \begin{gather*}
    w_G(A') \leq 2(1+c)\sqrt{2n\log(\lambda(B')/(2c)^n)}.
  \end{gather*}
  Finally,
  $w_G(A) = w_G(A')/2$ and $\lambda(A+[0,c]^n) = \lambda(B')/2^n$, whence the
  desired inequality.
\end{proof}

\begin{lem}
  \label{lem:GaussianMeanWidthSizeBound}
  Let $A \subseteq [0,1]^n$ be finite, $|A| = N$.
  Then $w_G(A) \leq \sqrt{2n\log N}$.
\end{lem}
\begin{proof}
  For $c$ small enough we have $\lambda(A+[0,c]^n) = Nc^n$, so
  $w_G(A) \leq (1+c)\sqrt{2n\log N}$ and
  thus $w_G(A) \leq \sqrt{2n\log N}$.
\end{proof}
Our proof of \fref{lem:GaussianMeanWidthMeasureBound} is based on a
direct argument due to M.\ Talagrand for
\fref{lem:GaussianMeanWidthSizeBound}.

\begin{thm}
  \label{thm:VCEquiv}
  Let $Q \subseteq [0,1]^X$ be a collection of functions.
  Then the following are equivalent:
  \begin{enumerate}
  \item \label{item:VCEquivDiagonal}
    $Q$ is a Vapnik-Chervonenkis class.
  \item \label{item:VCEquivNonDiagonal}
    For every $\varepsilon > 0$ there is $d$
    such that for every $\bar x \in X^d$,
    $Q(\bar x)$ does \emph{not} determine a $d$-dimensional $\varepsilon$-box.
  \item \label{item:VCEquivWidth}
    $\lim \frac{g_Q(n)}{n}  = 0$.
  \end{enumerate}
\end{thm}
\begin{proof}
  For \fref{item:VCEquivDiagonal} $\Longrightarrow$ \fref{item:VCEquivNonDiagonal}
  we shall prove the contra-positive.
  So assume that for some $\varepsilon> 0$ this fails, i.e., for all $d$ there
  are $\bar x \in X^d$, $r_0,\ldots,r_{d-1}$ and
  $\big\{q_w\colon w \subseteq d\big\} \subseteq Q$ satisfying
  $q_w(x_i) \leq r_i$ if $i \in w$ and
  $q_w(x_i) \geq r_i+\varepsilon$ if $i \notin w$.
  Thus there must be a subset of $r_i$ of size at least
  $d' = \lceil d\varepsilon/2 \rceil$ which are at distance at most $\varepsilon/2$ from one
  another, and we might as well assume these are
  $r_0 \leq r_1 \leq \ldots \leq r_{d'-1} \leq r_0 + \varepsilon/2 = r$.
  For $i < d'$ we $q_w(x_i) \leq r$ if $i \in w$ and
  $q_w(x_i) \geq r+\varepsilon/2$ if $i \notin w$.
  This works for arbitrarily large $d$, and thus for arbitrarily large
  $d'$.
  Thus $Q$ is not a Vapnik-Chervonenkis class.
  (And considering $d' = \lfloor d\varepsilon/m \rfloor$ we can get
  $q_w(x_i) \geq r+\varepsilon(1-\frac{1}{m})$.)

  Let us now show \fref{item:VCEquivNonDiagonal}
  $\Longrightarrow$ \fref{item:VCEquivWidth}.
  Let us fix $\varepsilon > 0$, and let $d$ be as in the hypothesis.
  By \fref{lem:VCSBox} and \fref{lem:GaussianMeanWidthMeasureBound}
  we have for all $c > 0$ and $\bar x \in X^n$:
  \begin{gather*}
    \lambda(Q(\bar x) + [0,c]^n) \leq ( c+\varepsilon )^{n-d} (1-\varepsilon)^d p_d(n), \\
    w_G(Q(\bar x)) \leq (1+c)\sqrt{2n\log(\lambda(Q(\bar x)+[0,c]^n)/c^n)}
  \end{gather*}
  Whereby:
  \begin{gather*}
    g_Q(n) \leq (1+c)\sqrt{2n\log\left(
        \left( 1+\frac{\varepsilon}{c} \right)^n
        \left( \frac{1-\varepsilon}{c+\varepsilon} \right)^d p_d(n)
      \right)}, \\
    \frac{g_Q(n)}{n}
    \leq (1+c)\sqrt{2\log\left( 1+\frac{\varepsilon}{c} \right)
      + \frac{1}{n} \log\left(
        \left( \frac{1-\varepsilon}{c+\varepsilon} \right)^d p_d(n)
      \right)}.
  \end{gather*}
  As $n$ goes to infinity the second term under the root disappears.
  In addition we have
  $\log(1+\varepsilon/c) \leq \varepsilon/c$, and we obtain:
  $\varlimsup \frac{g_Q(n)}{n} \leq (1+c)\sqrt{\frac{2\varepsilon}{c}}$.
  Minimum is reached when $c = 1$ in which case
  $\varlimsup \frac{g_Q(n)}{n} \leq \sqrt{8\varepsilon}$.
  This holds for every $\varepsilon > 0$, whereby
  $\lim \frac{g_Q(n)}{n} = 0$ as desired.

  \fref{item:VCEquivWidth} $\Longrightarrow$ \fref{item:VCEquivDiagonal}
  was proved in \fref{prp:MeanWidthImpliesVC}.
\end{proof}

Notice that the proof also tells us in fact something more precise:
\begin{cor}
  Assume that $\varlimsup \frac{g_Q(n)}{n} = C > 0$.
  Then for some $r$ the class
  $Q_{r,r+C^2/8}$ is \emph{not} a Vapnik-Chervonenkis class.
\end{cor}

In case $Q \subseteq \{0,1\}^X$, i.e., for collection of
characteristic functions, a box if exists has size one, so we can get
better bounds .
\begin{prp}
  \label{prp:CharacteristicVC}
  Let $\cC \subseteq \cP(X)$, $Q = \{\chi_C\colon C \in \cC\} \subseteq [0,1]^X$.
  Then $g_Q(n) \leq \sqrt{2n\log p_d(n)}$ and for $n$ big enough
  $g_Q(n) \leq \sqrt{2dn\log n}$, where
  $d = VC(\cC)$.
\end{prp}
\begin{proof}
  Let $d = VC(\cC) < \infty$.
  For every $\bar x \in X^n$ we have
  $|Q(\bar x)| \leq f_\cC(n) \leq p_d(n)$ by
  \fref{fct:VapnikChervonenkisIndexBound}.
  For $n$ large enough we have $p_d(n) \leq n^d$ and by
  \fref{lem:GaussianMeanWidthSizeBound}:
  \begin{gather*}
    g_Q(n) \leq \sqrt{2n\log p_d(n)} \leq \sqrt{2dn\log n}.
    \qedhere
  \end{gather*}
\end{proof}

We can now switch to a more symmetric situation.
Let $X$ and $Y$ be two sets, $S \subseteq X\times Y$.
For $x \in X$ let
$S_x = \{y \in Y\colon (x,y) \in S\}$ and for $y \in Y$ let
$S^y = \{x \in X\colon (x,y) \in S\}$.
Thus $S$ gives rise to two families of subsets
$S^Y = \{S^y\colon y \in Y\} \subseteq \cP(X)$ and
$S_X = \{S_x\colon x \in X\} \subseteq \cP(Y)$.

Similarly, if $S \sqsubseteq X\times Y$
we may define $S^y \sqsubseteq X$ by $x \in S^y \Longleftrightarrow (x,y) \in S$ and
$x \notin S^y \Longleftrightarrow (x,y) \notin S$.
Continuing as above we obtain two families
of fuzzy subsets $S^Y \subseteq \sfP(X)$ and $S_X \subseteq \sfP(Y)$.

\begin{fct}
  \label{fct:VapnikChervonenkisSymmetry}
  Let $S \sqsubseteq X \times Y$.
  Then $S_X$ is a Vapnik-Chervonenkis class if and only if $S^Y$ is,
  in which case $VC(S_X) \leq 2^{VC(S^Y)}$ and vice versa.

  We say in this case that $S$ is a \emph{dependent} relation.
\end{fct}
\begin{proof}
  In case $S \subseteq X \times Y$ this is proved in
  \cite[Chapter~5]{vandenDries:oMinimal}.
  The case of a fuzzy relation, while not considered there, is
  identical.
\end{proof}

Finally, a function $\varphi\colon X\times Y \to [0,1]$ gives rise to two families 
of functions
$\varphi^Y = \{\varphi^y\colon y \in Y\} = \{\varphi(\cdot,y)\colon y \in Y\} \subseteq [0,1]^X$ and similarly
$\varphi_X = \{\varphi_x\colon x \in X\} \subseteq [0,1]^Y$.

\begin{prp}
  \label{prp:ContinuousVapnikChervonenkisSymmetry}
  Let $X$ and $Y$ be sets, $\varphi\colon X\times Y \to [0,1]$ any function.
  Then $\varphi^Y$ is a Vapnik-Chervonenkis class if and only if
  $\varphi_X$ is.

  In that case we say that $\varphi$ is \emph{dependent}.
\end{prp}
\begin{proof}
  For $0\leq r<s\leq1$ define $\varphi_{r,s} \sqsubseteq X\times Y$ as usual.
  Then $(\varphi_{r,s})_X = (\varphi_X)_{r,s}$ is a Vapnik-Chervonenkis class if
  and only if $(\varphi_{r,s})^Y = (\varphi^Y)_{r,s}$ is.
\end{proof}

\begin{lem}
  \label{lem:DependentLimit}
  A uniform limit of dependent functions is dependent.
\end{lem}
\begin{proof}
  Let $\varphi_n\colon X\times Y \to [0,1]$ be dependent converging uniformly to $\varphi$.
  Assume $\varphi$ is independent, so say $\varphi_{r,r+3\varepsilon}$ is independent for
  some $\varepsilon > 0$ and $r \in [0,1-3\varepsilon]$.
  Let $n$ be large enough such that $|\varphi-\varphi_n| \leq \varepsilon$.
  Then $(\varphi_n)_{r+\varepsilon,r+2\varepsilon}$ is independent, contrary to hypothesis.
\end{proof}

\section{Crushing convex compacts}
\label{sec:Crush}

Let $K \subseteq \setR^n$ be a convex compact, $u \in S^{n-1}$ a fixed direction
vector.
We would like to construct a new convex compact $K_u$ by crushing all
points below the hyperplane $u^\perp$ to the hyperplane.
Define the two half spaces and a mapping $S\colon \setR^n \to \setR^n$ as follows:
\begin{gather*}
  H^+ = \{x \in \setR^n\colon \langle x,u\rangle \geq 0\}, \qquad
  H^- = \{x \in \setR^n\colon \langle x,u\rangle \leq 0\}, \\
  S(x) =
  \begin{cases}
    x & x \in H^+ \\
    P_{u^\perp}(x) & x \in H^-.
  \end{cases}
\end{gather*}
We then let
\begin{gather*}
  K_u = \Conv(S(K)) = \Conv((K\cap H^+) \cup P_{u^\perp}(K\cap H^-)).
\end{gather*}

We would like to show that $w_G(K_u) \leq w_G(K)$.

\begin{lem}
  \label{lem:CrushedHeightProperties}
  Let $K$, $u$ and $K_u$ be as above.
  If $h_K(-u) \leq 0$ then $K = K_u$.
  If $h_K(-u) \geq 0$ then we have
  for $y \in \setR^n$, $y' = P_{u^\perp}(y)$:
  \begin{align*}
    h_{K_u}(y) & = \max(h_K(y),h_K(P_{u^\perp}(y))) && y \in H^+, \\
    h_{K_u}(y) & \leq \min \{ h_K(z)\colon z \in [y,P_{u^\perp}(y)] \} && y \in H^-.
  \end{align*}
\end{lem}
\begin{proof}
  If $h_K(-u) \leq 0$ then $K  \subseteq H^+$ and $S(K) = K$.
  Consider the case $h_K(-u) \geq 0$.
  In that case clearly $h_{K_u}(u) = 0$ and
  $h_{K_u}$ agrees with $h_K$ on $u^\perp$.
  Let us also observe that if $y \in \setR^n$ then the 
  $h_{K_u}(y) = \langle x,y\rangle$ for some extremal point $x \in K_u$, in which case
  we have in fact $x \in S(K)$.
  Thus we always have $h_{K_u}(y) = \langle S(x),y\rangle$, $x \in K$.

  Let us consider the case where $y \in H^+$, i.e.,
  $y = y' + \lambda u$ where $y' \perp u$ and $\lambda \geq 0$.
  Say $h_K(y) = \langle x,y\rangle$, $x \in K$.
  Then $\langle S(x),y\rangle \geq \langle x,y\rangle$ and thus $h_{K_u} \geq h_K(y)$.
  Since $h_{K_u}$ is sub-additive we also have
  $h_{K_u}(y) = h_{K_u}(y'+\lambda u) + h_{K_u}(-\lambda u) \geq h_{K_u}(y')$.
  Thus $h_{K_u}(y) \geq \max(h_K(y),h_K(y'))$.
  On the other hand, we know that $h_{K_u}(y) = \langle S(x),y\rangle$ for some
  $x \in K$.
  If $x \in H^+$ then $h_{K_u}(y) \leq h_K(y)$.
  If $x \in H^-$
  then $\langle S(x),y\rangle = \langle S(x),y'\rangle = \langle x,y'\rangle$ so
  $h_{K_u}(y) \leq h_K(y')$.
  Either way $h_{K_u}(y) \leq \max(h_K(y),h_K(y'))$ and the first case is proved.

  Now assume $y \in H^-$.
  Let us make first some general observations.
  First, if $h_{K_u}(y) = \langle S(x),y\rangle$, $x \in K$, then
  $\langle S(x),y\rangle \leq \langle x,y\rangle$ whereby $h_{K_u}(y) \leq h_K(y)$.
  Now write $y = y' - \lambda u$ where $y' \perp u$ and
  $\lambda \geq 0$.
  Let $z \in [y,y'] \subseteq H^-$,
  i.e., $z = y'-\mu u$ for $\mu \in [0,\lambda]$.
  Then $h_{K_u}(y) \leq h_{K_u}(z) + h_{K_u}(-(\lambda-\mu)u) = h_{K_u}(z) \leq h_K(z)$.
  We have thus shown that
  $h_{K_u}(y) \leq \min \{ h_K(z) \colon z \in [y',y] \}$.
\end{proof}

Can the second inequality be improved to an equality?
Either way, the inequalities we have suffice to prove:

\begin{lem}
  Let $K$ and $K_u$ be as above, $y' \in u^\perp$ and
  $\lambda \geq 0$.
  Then
  $$h_{K_u}(y' + \lambda u) + h_{K_u}(y' - \lambda u)
  \leq h_K(y' + \lambda u) + h_K(y' - \lambda u).$$
\end{lem}
\begin{proof}
  Consider the mapping $s(t) = h_{K_u}(y'+tu)$, which we know to be
  convex.
  If $s(0) \leq s(\lambda)$ then
  $h_{K_u}(y' + \lambda u) = h_K(y' + \lambda u)$, and we already know that
  $h_{K_u}(y' - \lambda u) \leq h_K(y' - \lambda u)$.

  If $s(0) \geq s(\lambda)$ then $h_{K_u}(y' + \lambda u) = h_K(y')$.
  By convexity of $s$ it must be decreasing
  for all $t \leq 0$, so in particular
  $$h_{K_u}(y-\lambda u) \leq \min\{s(t)\colon t\in[-\lambda,0]\} = s(0) = h_K(y').$$
  Thus:
  \begin{gather*}
    h_{K_u}(y' + \lambda u) + h_{K_u}(y' - \lambda u)
    \leq h_K(2y') \leq h_K(y' + \lambda u) + h_K(y' - \lambda u).
    \qedhere
  \end{gather*}
\end{proof}

\begin{prp}
  \label{prp:CrushedMeanWidth}
  Let $K$ and $K_u$ be as above.
  Then $w(K_u) \leq w(K)$ and $w_G(K_u) \leq w_G(K)$.
\end{prp}
\begin{proof}
  It will be enough to prove the first inequality.
  For $y \in S^{n-1}$ let $y'$ always denote $P_{u^\perp}(y)$ and
  $\lambda = |\langle y,u\rangle|$.
  We have:
  \begin{align*}
    2w(K_u)
    & = 2\int_{S^{n-1}} w(K_u,y) \, d\sigma(y) \\
    & = \int_{S^{n-1}} (w(K_u,y'+\lambda u) + w(K_u,y'-\lambda u)) \, d\sigma(y) \\
    & = \int_{S^{n-1}}
    \begin{aligned}[t]
      & \big( h_{K_u}(y'+\lambda u) + h_{K_u}(y'-\lambda u) \\
      & \qquad + h_{K_u}(-y'+\lambda u) + h_{K_u}(-y'-\lambda u) \big) \, d\sigma(y)
    \end{aligned} \\
    & \leq \int_{S^{n-1}}
    \begin{aligned}[t]
      & \big( h_K(y'+\lambda u) + h_K(y'-\lambda u) \\
      & \qquad + h_K(-y'+\lambda u) + h_K(-y'-\lambda u) \big) \, d\sigma(y)
    \end{aligned} \\
    & = \ldots = 2w(K).
    \qedhere
  \end{align*}
\end{proof}

Now let $K \subseteq \setR^n$ be a convex compact,
$(e_i\colon i < n)$ the canonical base, and define
\begin{align*}
  K^+ & = (\ldots(K_{e_0})_{e_1}\ldots)_{e_{n-1}} \\
  & = \Conv\big( (x_0\lor0,\ldots,x_{n-1}\lor0)\colon \bar x \in K \big).
\end{align*}

\begin{cor}
  \label{cor:PositiveMeanWidth}
  Let $K \subseteq \setR^n$ be a convex compact.
  Then $w_G(K^+) \leq w_G(K)$.
\end{cor}

We remind the reader that for $x,y \in [0,1]$ we define
$\lnot x = 1-x \in [0,1]$ and $x \dotminus y = \max(x-y,0) \in [0,1]$.
Moreover, for every $n \geq 1$, the family of functions
$[0,1]^n \to [0,1]$ one can construct with the three operations
$\{\half[x],\lnot x,x \dotminus y\}$ is dense in the space all continuous
functions from $[0,1]^n$ to $[0,1]$
(see \cite{BenYaacov-Usvyatsov:CFO}).

\begin{cor}
  \label{cor:CombinationMeanWidth}
  Let $X$ and $Y$ be sets, $\varphi,\psi\colon X\times Y \to [0,1]$.
  Then $g_{(\lnot \varphi)^Y} = g_{\varphi^Y}$, $g_{(\varphi/2)^Y} = \half g_{\varphi^Y}$ and
  $g_{(\varphi\dotminus \psi)^Y} \leq g_{\varphi^Y} + g_{\psi^Y}$.

  Thus, if $\varphi$ and $\psi$ are Vapnik-Chervonenkis classes then so are
  $\lnot \varphi$, $\half \varphi$ and $\varphi \dotminus \psi$.
\end{cor}
\begin{proof}
  Clearly $g_{(\lnot \varphi)^Y} = g_{\varphi^Y}$, $g_{(\varphi/2)^Y} = \half g_{\varphi^Y}$.
  We are left with $g_{(\varphi\dotminus \varphi)^Y} \leq g_{\varphi^Y} + g_{\psi^Y}$.

  Consider the function $\varphi - \psi \colon X\times Y \to [-1,1]$, and observe that for
  $\bar x \in X^n$ we have
  $(\varphi\dotminus\psi)^Y(\bar x) = \left( (\varphi-\psi)^Y(\bar x) \right)^+ \subseteq
  [0,1]^n$.
  We thus have:
  \begin{gather*}
    w_G\left( (\varphi\dotminus\psi)^Y(\bar x) \right)
    = w_G\left( \left( (\varphi-\psi)^Y(\bar x) \right)^+ \right)
    \leq w_G\left( (\varphi-\psi)^Y(\bar x) \right).
  \end{gather*}
  On the other hand we also have
  $(\varphi-\psi)^Y \subseteq \varphi^Y - \psi^Y \subseteq [0,1]^X$, and for $\bar x \in X^n$:
  \begin{gather*}
    w_G\left( (\varphi-\psi)^Y(\bar x) \right)
    \leq w_G\left( \varphi^Y(\bar x) - \psi^Y(\bar x) \right)
    = w_G\left( \varphi^Y(\bar x) \right) + w_G\left( \psi^Y(\bar x) \right).
  \end{gather*}
  Thus $w_G\left( (\varphi\dotminus\psi)^Y(\bar x) \right)
  \leq w_G\left( \varphi^Y(\bar x) \right) + w_G\left( \psi^Y(\bar x) \right)$,
  whereby $g_{(\varphi\dotminus\psi)^Y} \leq g_{\varphi^Y} + g_{\psi^Y}$.
\end{proof}

\begin{lem}
  \label{lem:DependentCombinations}
  Let $X$ and $Y$ be sets, $\varphi,\psi\colon X\times Y \to [0,1]$ dependent.
  Then $\lnot \varphi$, $\half \varphi$ and $\varphi \dotminus \psi$ are dependent as well.
\end{lem}
\begin{proof}
  By \fref{cor:CombinationMeanWidth}.
\end{proof}

\begin{prp}
  \label{prp:DependentCombinations}
  Let $X$ and $Y$ be sets, $\varphi_n\colon X \times Y \to [0,1]$ dependent
  functions for $n < \omega$,
  and let $\psi\colon [0,1]^\omega \to [0,1]$ be an arbitrary continuous function.
  Then $\psi \circ (\varphi_n)\colon X\times Y \to [0,1]$ is dependent.
\end{prp}
\begin{proof}
  By results in \cite{BenYaacov-Usvyatsov:CFO} one can approximate $\psi$
  uniformly with expressions written with $\lnot$, $\half$ and
  $\dotminus$.
  Such expressions in the $\varphi_n$ are dependent by
  \fref{lem:DependentCombinations}.
  Thus $\psi \circ (\varphi_n)$ is a uniform limit of dependent functions, and is
  therefore dependent by \fref{lem:DependentLimit}.
\end{proof}

\section{Random dependent relations and functions}
\label{sec:Randomise}

In this section $X$ and $Y$ will be sets as before.
However, we will be interested here in dependent relations and
functions on $X \times Y$ which may vary (randomly).

Let $\Omega$ be an arbitrary set for the time being.
A family of relations on $X \times Y$, indexed by $\Omega$, can be viewed as a
relation $S \subseteq \Omega \times X \times Y$.
For every $\omega \in \Omega$ we obtain a relation $S_\omega \subseteq X \times Y$ and we may view
$S$ equivalently as a function $S\colon \Omega \to \cP(X \times Y)$.
Similarly, a family of $[0,1]$-valued functions on $X \times Y$ will be
given as $\varphi\colon \Omega \times X \times Y \to [0,1]$ or equivalently as
$\varphi \colon \Omega \to [0,1]^{X\times Y}$ sending $\omega \mapsto \varphi_\omega = \varphi(\omega,\cdot,\cdot)$.
The usual passage from $S$ to its characteristic function $\chi_S$
commutes with these equivalent presentations.

We say that such a family $S = \{S_\omega\colon \omega \in \Omega\}$ is \emph{uniformly dependent}
if there is $d = d(S)$ such that $VC((S_\omega)^Y) \leq d$ for every $\omega \in \Omega$.
Similarly a family $\varphi = \{\varphi_\omega\colon \omega \in \Omega\}$ is uniformly dependent if
for every $\varepsilon > 0$ there is $d = d(\varphi,\varepsilon)$ such that
$VC\left( (\varphi_\omega^Y)_{[r,r+\varepsilon]} \right) \leq d$ for every
$r \in [0,1-\varepsilon]$ and $\omega \in \Omega$.
Clearly $S$ is uniformly dependent if and only if $\chi_S$ is.

It follows from the proof of \fref{thm:VCEquiv}
that $\varphi = \{\varphi_\omega\colon \omega \in \Omega\}$ is uniformly dependent
if and only if there is a function
$g\colon \setN \to \setR$ such that
$\lim \frac{g(n)}{n} = 0$ and
$g_{\varphi_\omega^Y} \leq g$ for every $\omega$.
Indeed, in case $\varphi$ is uniformly dependent then for every
$\omega$, $n$ and $\varepsilon$ we obtain:
\begin{gather*}
  g_{\varphi_\omega^Y}(n) \leq 
  2n \sqrt{
    2\varepsilon + \frac{d(\varphi,\varepsilon)}{n} \log\left( \frac{1-\varepsilon}{1+\varepsilon} \right)
    + \frac{\log p_{d(\varphi,\varepsilon)}(n)}{n}
  }.
\end{gather*}
Then a function $g$ as desired can be obtained by:
\begin{gather*}
  g(n) = \inf_{0 < \varepsilon < 1}
  2n \sqrt{
    2\varepsilon + \frac{d(\varphi,\varepsilon)}{n} \log\left( \frac{1-\varepsilon}{1+\varepsilon} \right)
    + \frac{\log p_{d(\varphi,\varepsilon)}(n)}{n}
  }.
\end{gather*}

Let us now consider \emph{random}
relations and functions on $X \times Y$.
We fix a probability space $(\Omega,\fB,\mu)$.
From now on we will only consider families
$S$ or $\varphi$ such that for $(x,y) \in X \times Y$ the event
$\{\omega\colon (x,y) \in S_\omega\}$ or the function
$\omega \mapsto \varphi_\omega(x,y)$ are measurable.
We may then define functions
$\bP[S],\bE[\varphi]\colon X\times Y \to [0,1]$ by
\begin{gather*}
  \bP[S](x,y) = \bP[(x,y) \in S], \qquad
  \bE[\varphi](x,y) = \bE[\varphi(x,y)].
\end{gather*}
If $S$ is measurable then so is
$\chi_S$ which is given by $(\chi_S)_\omega = \chi_{(S_\omega)}$
and then $\bE[\chi_S] = \bP[S]$.

\begin{thm}
  \label{thm:RandomCountableFuncVC}
  Let $X$, $Y$ be countable sets, $\varphi_\omega\colon X \times Y \to [0,1]$ a random family of
  functions on $X\times Y$.
  Then $g_{\bE[\varphi]^Y} \leq \bE[g_{\varphi_\omega^Y}]$ (and the latter is measurable).

  In particular, if $\varphi$ is uniformly dependent then $\bE[\varphi]$ is dependent.
\end{thm}
\begin{proof}
  Let us fix $n$ and let $\bar x \in X^n$.
  Define
  \begin{gather*}
    \KK_{\bar x}(\omega) = \overline{\Conv}\big( \varphi_\omega^Y(\bar x) \big) \subseteq [0,1]^n.
  \end{gather*}
  Each $\KK_{\bar x}(\omega)$ is a convex compact and
  \begin{gather*}
    g_{\varphi^Y_\omega}(n) = \sup_{\bar x \in X^n} w_G(\KK_{\bar x}(\omega)).
  \end{gather*}
  Since $Y$ is assumed countable the family $\KK_{\bar x}$ is measurable for
  every $\bar x \in X$.
  It is moreover bounded and therefore integrable.
  Since $X$ is also assumed countable the function $\omega \mapsto g_{\varphi^Y_\omega}(n)$
  is measurable as well.
  For a fixed tuple $\bar x$ we
  have $\bE[\varphi]^Y(\bar x) \subseteq \bE[\KK_{\bar x}]$.
  Thus
  \begin{gather*}
    w_G\left( \bE[\varphi]^Y(\bar x) \right)
    \leq w_G(\bE[\KK]_{\bar x})
    = \bE[ w_G(\KK_{\bar x}) ] \leq \bE[g_{\varphi_\omega^Y}(n)].
  \end{gather*}
  It follows that $g_{\bE[\varphi]^Y} \leq \bE[g_{\varphi_\omega^Y}]$ as desired.

  If $\varphi$ is uniformly dependent then there is $g\colon \setN \to \setR$ such that
  $\frac{g(n)}{n} \to 0$ and
  $g \geq g_{\varphi_\omega^Y}$ for every $\omega \in \Omega$.
  Then $g_{\bE[\varphi]^Y} \leq g$ as well and by \fref{thm:VCEquiv}
  $\bE[\varphi]$ is dependent.
\end{proof}

\begin{cor}
  \label{cor:RandomFuncVC}
  Let $X$, $Y$ be sets, $\varphi = \{\varphi_\omega\colon \omega \in \Omega\}$ a measurable family of
  uniformly dependent functions.
  Then $\bE[\varphi] \colon X \times Y \to [0,1]$ is dependent.
\end{cor}
\begin{proof}
  If not then this is witnesses on countable subsets $X_0 \subseteq X$ and
  $Y_0 \subseteq Y$, contradicting \fref{thm:RandomCountableFuncVC}.
\end{proof}

\begin{cor}
  \label{cor:RandomSetVC}
  Let $X$, $Y$ be sets, $S = \{S_\omega\colon \omega \in \Omega\}$ a measurable family of
  uniformly dependent relations.
  Then $\bP[S] \colon X \times Y \to [0,1]$ is dependent.
\end{cor}
\begin{proof}
  Apply \fref{cor:RandomFuncVC} to $\chi_S$.
\end{proof}

\section{Dependent and randomised theories}
\label{sec:DependentTheories}

In this final section we settle the model theoretic problem which
motivated the present study.
This consists mostly of translating consequences
of previous results to the model theoretic setting.
In order to avoid blowing this section up disproportionately with a lot of
introductory material we assume the reader is already familiar
with the basics of
classical model theory (see Poizat \cite{Poizat:Cours}) and its
generalisation to continuous logic
(see \cite{BenYaacov-Usvyatsov:CFO}).

Let $T$ be a (classical or continuous) first order theory.
\begin{dfn}
  We say that a formula $\varphi(\bar x,\bar y)$ is \emph{dependent} in $T$
  if for every $\cM \models T$,
  $\varphi^\cM$ is dependent on $M^n \times M^m$.

  We say that $T$ is \emph{dependent} if all formulae are.
\end{dfn}
In the case of a classical theory this is equivalent to the original
definition (see Laskowski \cite{Laskowski:VapnikChervonenkisClasses})
and it extends naturally to continuous logic.
If $T$ is a continuous dependent theory then
by \fref{lem:DependentLimit} every definable predicate in $T$ is
dependent.
In addition, it is easy to see using compactness that if
$\varphi(\bar x,\bar y)$ is
dependent in $T$ then it is uniformly so in all models of $T$.

For a continuous language $\cL$ let
$\cL^R$ consists of a $n$-ary predicate symbol
$\bE[\varphi(\bar x)]$ for every $n$-ary $\cL$-formula $\varphi(\bar  x)$.

\begin{thm}
  For every $\cL$-theory $T$ (dependent or not)
  there is a (unique) $\cL^R$-theory $T^R$   such that:
  \begin{enumerate}
  \item For every $p(\bar x) \in \tS_n(T^R)$ there is a unique Borel
    probability measure $\nu_p$ on $\tS_n(T^R)$ such that for every
    $n$-ary predicate symbol $\bE[\varphi(\bar x)] \in \cL^R$:
    \begin{gather*}
      \bE[\varphi(\bar x)]^p = \int \varphi^q\,d\nu_p(q).
    \end{gather*}
    The mapping $p \mapsto \nu_p$ is a bijection between
    $\tS_n(T^R)$ and the space of regular Borel probability measures.
    We will consequently identify the two spaces,
    thus identifying $p$ with $\nu_p$.
  \item The topology on $\tS_n(T^R)$ is the one of weak convergence.
    In other words, this is the weakest topology
    such that for every continuous function
    $\varphi\colon \tS_n(T) \to \setC$ the mapping
    $\mu \mapsto \int \varphi \, d\mu$ is continuous.
  \item For a mapping $f\colon m \to n$, the corresponding mapping
    $f^{*,R}\colon \tS_n(T^R) \to \tS_m(T^R)$ is given by associating to each
    type in $\tS_n(T^R)$, being a measure on $\tS_n(T)$,
    its image measure on $\tS_m(T)$ via the application
    $f^*\colon \tS_n(T) \to \tS_m(T)$.
    (Since $f^*\colon \tS_n(T) \to \tS_m(T)$ is continuous between compact
    spaces, the image of a regular measure is regular.)
  \item The distance predicate coincides with
    $\bE[d(x,y)]$.
  \end{enumerate}
  Moreover, $T^R$ eliminates quantifiers.

  Since every classical first order theory can be viewed as a
  continuous theory, the same applies if $T$ is a classical theory.
  In this case we may prefer to write $\bP[\varphi(\bar x)]$ instead of
  $\bE[\varphi(\bar x)]$.
  (In fact, the precise counterpart of $\bE[\varphi(\bar x)]$ is
  $\bP[\lnot\varphi(\bar x)]$ since $1$ is ``False'', but this is a minor issue.)
  In particular the distance predicate is then given by
  $\bP[x\neq y]$.
\end{thm}
\begin{proof}
  Uniqueness follows from the fact that the type spaces are entirely
  described.

  In the case $T$ is a classical theory, the explicit construction
  appears in \cite{BenYaacov-Keisler:MetricRandom}, where
  Keisler's original construction
  \cite{Keisler:Randomizing} is transferred
  from classical logic to the more adequate setting of
  continuous logic.

  A similar construction can in principle
  be carried out when $T$ is a continuous
  theory.
  Alternatively, let us consider $\tS_n(T^R)$ as a mere symbol
  denoting   the space of regular Borel
  probability measures on $\tS_n(T)$.
  Let $\tS(T^R)$ denote the mapping $n \mapsto \tS_n(T^R)$
  and let us equip it with the topological and
  functorial structure described in items (ii),(iii).
  Then $\tS(T^R)$ is an open Hausdorff
  type-space functor in the sense of
  \cite{BenYaacov:PositiveModelTheoryAndCats}.
  The predicates of $\cL^R$ can be interpreted in models of $\tS(T^R)$
  as per item (i), in which case $\bE[d(x,y)]$
  defines a metric on the models.
  By results appearing in \cite{BenYaacov-Usvyatsov:CFO} a continuous
  theory $T^R$ exists in \emph{some} language whose type space functor
  is $\tS_n(T^R)$.
  Since the $n$-ary $\cL$-formulae are dense among all
  continuous functions $\tS_n(T) \to [0,1]$,
  the atomic $\cL^R$-formulae $\bE[\varphi(\bar x)]$ separate types.
  It follows that $T^R$ can be taken to be an $\cL^R$-theory and that
  it eliminates quantifiers as such.
  We leave the details to the reader.
\end{proof}

Members of models of $T^R$  should be thought of as random
variables in models of $T$.
If ${\mathbf a},{\mathbf b},\ldots \in \cM \models T^R$
then their type
$\tp^R({\mathbf a},{\mathbf b},\ldots)$,
viewed as a probability measure,
should be thought of as the distribution measure of
the $\tS_n(T)$-valued random variable
$\omega \mapsto \tp({\mathbf a}(\omega),{\mathbf b}(\omega),\ldots)$.
Similarly
$\bE[\varphi({\mathbf a},{\mathbf b},\ldots)]$ is the expectation
of the random variable 
$\omega \mapsto \varphi({\mathbf a}(\omega),{\mathbf b}(\omega),\ldots)$,
and so on.
As we said in the introduction it is natural to ask whether
the randomisation of a dependent theory is dependent.

\begin{thm}
  \label{thm:RandomDependentTheory}
  Let $T$ be a dependent first order theory (classical or continuous).
  Then $T^R$ is dependent as well.
\end{thm}
\begin{proof}
  Every classical theory can be identified with a continuous theory
  via the identification of  $T$ with $0$, of $F$ with $1$ and
  of $=$ with $d$.
  We may therefore assume that $T$ is continuous.

  Let us first consider a formula of the form
  $\varphi(\bar x,\bar y) = \bE[\psi(\bar x,\bar y)]$.
  Let $\cM \models T^R$, and we need to show that $\varphi^\cM$ is dependent on
  $M^n \times M^m$.
  Let us enumerate
  $M^n = \{\bar a_i\colon i \in I\}$, $M^m = \{\bar b_j\colon j \in J\}$.
  Let $p = \tp(M^n , M^m/\emptyset)$.
  We may write it as $p( \bar x_i, \bar y_j )_{i\in I,j \in J} \in \tS_{I \cup J}(T^R)$,
  and identify it with a probability measure
  $\mu$ on $\Omega = \tS_{(I\times n) \cup (J\times m)}(T)$ such that for every formula $\rho(\bar z)$
  of the theory $T$, $\bar z \subseteq \{\bar x_i,\bar y_j\}_{i\in I,j\in J}$:
  \begin{gather*}
    \bE[\rho(\bar z)]^p = \int_\Omega \rho(\bar z)^q\, d\mu(q).
  \end{gather*}
  For $i \in I$, $j \in J$ and $q \in \Omega$ define:
  $\chi_q(i,j) = \psi(\bar x_i,\bar y_j)^q$.
  Then $\chi = \{\chi_q\colon q \in \Omega\}$ is a measurable family of
  $[0,1]$-valued functions on $I \times J$
  and $\varphi(\bar a_i,\bar b_j) = \varphi(\bar x_i,\bar y_j)^p = \bE[\chi](i,j)$
  where expectation is
  with respect to $\mu$.
  Since $T$ is dependent the family $\{\chi_q\colon q \in \Omega\}$ is uniformly
  dependent.
  By \fref{cor:RandomFuncVC} $\bE[\chi] \colon I\times J \to [0,1]$ is dependent.
  Equivalently, $\varphi\colon M^n \times M^m \to [0,1]$ is dependent.

  We have thus shown that every atomic formula is dependent.
  By   \fref{lem:DependentCombinations} every quantifier free formula
  is dependent.
  By quantifier elimination and \fref{lem:DependentLimit} every
  formula is dependent.
\end{proof}

We conclude this paper with a few extensions of classical results
regarding dependent formulae and theories to continuous logic.

\begin{lem}
  \label{lem:IndependentByIndiscernible}
  The following are equivalent for a formula $\varphi(\bar x,\bar y)$:
  \begin{enumerate}
  \item The formula $\varphi$ is independent.
  \item There exist a tuple $\bar a$ an, indiscernible sequence
    $(\bar b_n\colon n < \omega)$ and $0 \leq r < s \leq 1$ such that:
    \begin{gather*}
      \varphi(\bar a,\bar b_{2n}) \leq r, \qquad
      \varphi(\bar a,\bar b_{2n+1}) \geq s.
    \end{gather*}
  \item There exist a tuple $\bar a$ and indiscernible sequence
    $(\bar b_n\colon n < \omega)$ such that $\lim \varphi(\bar a,\bar b_n)$ does not exists.
  \end{enumerate}
\end{lem}
\begin{proof}
  \begin{cycprf}
  \item[\impnext]
    Assume $\varphi$ is independent, and let us work in a
    sufficiently saturated model.
    Then there are $0 \leq r < s\leq 1$ such that for all $m$ there are
    $(\bar b_n\colon n < m)$ and $(\bar a_w\colon w \subseteq m)$ satisfying:
    \begin{gather*}
      \varphi(\bar a_w,\bar b_n) \leq r \Longleftrightarrow n \in w, \qquad
      \varphi(\bar a_w,\bar b_n) \geq s \Longleftrightarrow n \notin w.
    \end{gather*}
    By compactness there exists an infinite sequence $(b_n\colon n < \omega)$
    such that for every finite $u \subseteq \omega$ and every $w \subseteq u$ there are
    $\bar a_{u,w}$ such that for all $n \in u$:
    \begin{gather*}
      \varphi(\bar a_{u,w},\bar b_n) \leq r \Longleftrightarrow n \in w, \qquad
      \varphi(\bar a_{u,w},\bar b_n) \geq s \Longleftrightarrow n \notin w.
    \end{gather*}
    By standard arguments using Ramsey's Theorem there exists an
    indiscernible sequence $(\bar b_n\colon n < \omega)$ having the same
    property.
    In particular for every $m$ there exists $\bar a_m$
    such that for all $n < m$:
    \begin{gather*}
      \varphi(\bar a,\bar b_{2n}) \leq r, \qquad
      \varphi(\bar a,\bar b_{2n+1}) \geq s.
    \end{gather*}
    The existence of $\bar a$ as desired now follows by compactness.
  \item[\impnext] Immediate.
  \item[\impfirst]
    Assume that $(\bar b_n\colon n < \omega)$ is indiscernible and
    $\lim_n \varphi(\bar a,\bar b_n)$ does not exist.
    Then there are $0 \leq r < s \leq 1$ such that
    $\varphi(\bar a,\bar b_n) < r$ and $\varphi(\bar a,\bar b_n) > s$ infinitely
    often.
    Then for every $m$ and every $w \subseteq m$ we can find
    $n_0 < \ldots < n_{m-1} < \omega$ such that
    $\varphi(\bar a,\bar b_{n_i}) < r$ if $i \in w$ and
    $\varphi(\bar a,\bar b_{n_i}) > s$ otherwise.
    By indiscernibility we can then find $\bar a_w$ such that
    $\varphi(\bar a,\bar b_i) < r$ if $i \in w$ and
    $\varphi(\bar a,\bar b_i) > s$ if $i \in m \setminus w$.
    Then $\varphi$ is independent.
  \end{cycprf}
\end{proof}

\begin{lem}
  \label{lem:PartialAverageSequenceType}
  Let $\bar a$ be a tuple,
  $(\bar b_n\colon n < \omega)$ an indiscernible sequence of tuples, and let
  $\varphi_s(\bar x,\bar y_0\ldots\bar y_{k_s-1})$ be dependent formulae for
  $s \in S$.
  Then there exists in an elementary extension of $\cM$
  an $\bar a$-indiscernible sequence
  $(\bar c_n\colon n < \omega)$ such that for all $s \in S$:
  \begin{gather*}
    \varphi_s(\bar a,\bar c_0\ldots\bar c_{k_s-1})
    =
    \lim \varphi_s(\bar a,\bar b_n\ldots\bar b_{n+k_s-1}).
  \end{gather*}
\end{lem}
\begin{proof}
  For $k < \omega$ let $I_k$ consist of all increasing tuples
  $\bar n \in \omega^k$.
  We define a partial ordering on $I_k$ saying that
  $\bar n < \bar n'$ if $n_{k-1} < n_0'$.
  Then standard arguments using Ramsey's Theorem and compactness yield
  an $\bar a$-indiscernible sequence $(\bar c_n\colon n < \omega)$ such that for every
  $k$ and every formula $\varphi(\bar x,\bar y_0,\ldots,\bar y_{k-1})$:
  \begin{gather*}
    \varliminf_{\bar n \in I_k} \varphi(\bar a,\bar b_{\bar n})
    \leq
    \varphi(\bar a,\bar c_0,\ldots,\bar c_{k-1})
    \leq
    \varlimsup_{\bar n \in I_k} \varphi(\bar a,\bar b_{\bar n}),
  \end{gather*}
  where $\bar b_{\bar n} = \bar b_{n_0},\ldots,\bar b_{n_{k-1}}$.

  Let us now fix $s \in S$.
  If $(\bar n_m\colon m < \omega)$ is an increasing sequence in $I_{k_s}$
  then $(\bar b_{\bar n_m}\colon m < \omega)$ is an indiscernible
  sequence so $\lim_m \varphi_s(\bar a,\bar b_{\bar n_m})$ exists.
  Moreover, given two increasing sequences in $I_{k_s}$
  we can choose a third increasing sequence alternating between the
  two, so the limit does not depend on the choice of sequence.
  It follows that $\lim_{\bar n \in I_{k_s}} \varphi_s(\bar a,\bar b_{\bar n})$
  exists, and the assertion follows.
\end{proof}

\begin{thm}
  Assume $T$ is independent.
  Then there exists a formula $\varphi(x,\bar y)$, where $x$ is a singleton,
  which is independent.
\end{thm}
\begin{proof}
  Let $\varphi(\bar x,\bar y)$ be an independent formula such that $\bar x$
  has minimal length.
  If it is of length one we are done.
  If not, we may write $\bar x = x\bar z$ and
  $\varphi = \varphi(x\bar z,\bar y)$.

  By \fref{lem:IndependentByIndiscernible}
  there are $a\bar b$ and a sequence $(\bar c_n\colon n < \omega)$
  is a model of $T$ as well as $0\leq r < s \leq 1$ such that
  \begin{gather*}
    \varphi(a\bar b,\bar c_{2n}) \leq r,\qquad
    \varphi(a\bar b,\bar c_{2n +1 }) \geq s.
  \end{gather*}
  Choosing $t \in (r,s)$ dyadic and replace
  $\varphi$ with $m(\varphi \dotminus t)$ for $m$ large enough we may assume that
  $r = 0$ and $s = 1$.
  For $m < \omega$ let:
  \begin{gather*}
    \psi_m(\bar z,\bar y_0,\ldots,\bar y_{2m-1})
    = \inf_x \bigvee_{i<m} (\varphi(x\bar z,\bar y_{2i}) \lor \lnot\varphi(x\bar z,\bar y_{2i+1})).
  \end{gather*}
  By assumption of minimality of $\bar x$ the formulae $\psi_n$ must be
  dependent.
  By \fref{lem:PartialAverageSequenceType} there is a
  $\bar b$-indiscernible sequence $(\bar c_n'\colon n < \omega)$ such that for
  all $m$:
  \begin{gather}
    \label{eq:IndependentSingleVariable}
    \psi_m(\bar b,\bar c_0'\ldots\bar c_{2m-1})
    =
    \lim \psi_m(\bar b,\bar c_n\ldots\bar c_{n+2m-1}).
  \end{gather}
  We know that $\psi_m(\bar b,\bar c_{2n},\ldots,\bar c_{2n+2m-1}) = 0$, as
  this is witnessed by $a$.
  Therefore the limit in \fref{eq:IndependentSingleVariable} must be
  equal to zero, and thus
  $\psi_m(\bar b,\bar c_0'\ldots\bar c_{2m-1}) = 0$ for all $m$.
  By a compactness argument there exists $a'$ such that
  \begin{gather*}
    \varphi(a'\bar b,\bar c'_{2n}) = 0, \qquad
    \varphi(a'\bar b,\bar c'_{2n+1}) = 1.
  \end{gather*}
  Changing our point of view a little we observe that
  $(\bar b\bar c_n'\colon n < \omega)$ is an indiscernible sequence and
  \begin{gather*}
    \varphi(a',\bar b\bar c'_{2n}) = 0, \qquad
    \varphi(a',\bar b\bar c'_{2n+1}) = 1.
  \end{gather*}
  Thus $\varphi(x,\bar z\bar y)$ is independent and $\bar x$ was not minimal
  after all.
\end{proof}

\providecommand{\bysame}{\leavevmode\hbox to3em{\hrulefill}\thinspace}
\providecommand{\MR}{\relax\ifhmode\unskip\space\fi MR }
\providecommand{\MRhref}[2]{%
  \href{http://www.ams.org/mathscinet-getitem?mr=#1}{#2}
}
\providecommand{\href}[2]{#2}

\end{document}